\documentclass[dvipdfmx,a4,11pt]{amsart}
\usepackage{amssymb,amsmath,amsthm}
\usepackage{url}
\usepackage{bm}
\usepackage{graphicx}
\usepackage{enumerate}
\usepackage{multicol}
\usepackage{multirow}
\usepackage{array}
\usepackage[all]{xy}
\usepackage{color}
\newtheorem{thm}{Theorem}[section]
\newtheorem{cor}[thm]{Corollary}
\newtheorem{lem}[thm]{Lemma}

\theoremstyle{definition}
\newtheorem{defin}[thm]{Definition}
\newtheorem{rem}[thm]{Remark}
\newtheorem{exa}[thm]{Example}

\numberwithin{equation}{section}

\author[M. Matsuno]{Masaki Matsuno}
\address{
	Graduate School of Integrated Science and Technology, 
	Shizuoka University\\
	Ohya 836, Shizuoka 422-8529, Japan}
\email{matsuno.masaki.14@shizuoka.ac.jp}

\keywords {Twisting systems, Twisted algebras, Geometric algebras, AS-regular algebras}

\subjclass[2020]{14A22, 16S38, 16W50.}

\thanks{The author was supported by Research Fellowships of the Japan Society for the Promotion of Science for Young Scientists
	(No. 21J11303).}

\newcommand{\PGl}{{\rm PGL}}
\newcommand{\Aut}{{\rm Aut}}

\newcommand{\GrAut}{{\rm GrAut}}

\newcommand{\Ext}{{\rm Ext}}
\newcommand{\Ker}{{\rm Ker\,}}

\newcommand{\ddPAut}{\Aut_{k}(\mathbb{P}^{2} \downarrow E)}
\newcommand{\uuPAut}{\Aut_{k}(E \uparrow \mathbb{P}^{2})}
\newcommand{\N}{N(E,\sigma)}
\newcommand{\M}{M(E,\sigma)}
\newcommand{\Z}{Z(E,\sigma)}
\newcommand{\TS}{{\rm TS}^{\mathbb{Z}}}
\newcommand{\TSo}{{\rm TS}^{\mathbb{Z}}_{0}}
\newcommand{\TSa}{{\rm TS}^{\mathbb{Z}}_{\rm alg}}

\newcommand{\GrMod}{{\rm GrMod}}

\newcommand{\al}{\alpha}
\newcommand{\be}{\beta}
\newcommand{\ga}{\gamma}

\newcommand{\la}{\lambda}
\newcommand{\ep}{\varepsilon}

\begin{document}
	\title[Twisted algebras of geometric algebras]
	{Twisted algebras of geometric algebras}
	\begin{abstract}
	A twisting system is one of the major tools to study graded algebras, however,
	it is often difficult to construct a (non-algebraic) twisting system if a graded algebra
	is given by generators and relations.
	In this paper, we show that a twisted algebra of a geometric algebra
	is determined by a certain automorphism
	of its point variety. As an application, we classify twisted algebras of $3$-dimensional geometric Artin-Schelter regular
	algebras up to graded algebra isomorphism.
	\end{abstract}
	\maketitle
	\section{Introduction}
	The notion of twisting system was introduced by Zhang in \cite{Z}.
	If there is a twisting system $\theta=\{ \theta_n \}_{n \in \mathbb{Z}}$ for a graded algebra $A$,
	then we can define a new graded algebra $A^{\theta}$, called a twisted algebra.
	Zhang gave a necessary and sufficient algebraic condition via a twisting system
	when two categories of graded right modules
	are equivalent (\cite[Theorem 3.5]{Z}).
	Although a twisting system plays an important role to study a graded algebra,
	it is often difficult to construct a twisting system on a graded algebra
	if it is given by generators and relations.
	
	Mori introduced the notion of geometric algebra $\mathcal{A}(E,\sigma)$
	which is determined by a geometric data
	which consists of a projective variety $E$, called the point variety, and its automorphism $\sigma$.
	For these algebras, Mori gave a necessary and sufficient geometric condition when two categories of graded right modules
	are equivalent (\cite[Theorem 4.7]{Mo}).
	By using this geometric condition, we can easily construct a twisting system.
	
	Cooney and Grabowski defined a groupoid whose objects are geometric noncommutative projective spaces
	and whose morphisms are isomorphisms between them.
	By studying a correspondence between the morphisms of this groupoid and a twisting system,
	they showed that the morphisms of this groupoid are parametrized by a set of certain automorphisms of the point variety
	(\cite[Theorem 28]{CG}).
	
	In this paper, we focus on studying a twisted algebra of a geometric algebra $A=\mathcal{A}(E,\sigma)$.
	For a twisting system $\theta$ on $A$,
	we set $\Phi(\theta):=\overline{(\theta_{1}|_{A_{1}})^{\ast}} \in \Aut_{k}\,\mathbb{P}(A_1^{\ast})$
	by dualization and projectivization.
	We find a subset $M(E,\sigma)$ of $\Aut_{k}\,\mathbb{P}(A_{1}^{\ast})$ parametrizing twisted algebras of $A$
	up to isomorphism.
	As an application to $3$-dimensional geometric Artin-Schelter regular algebras,
	we will compute $M(E,\sigma)$ (see Theorem \ref{thm.listNEC} and Theorem \ref{thm.listEC}),
	which completes the classification of twisted algebras of $3$-dimensional geometric Artin-Schelter regular algebras.
	
	Itaba and the author showed that for any $3$-dimensional quadratic Artin-Schelter regular algebra $B$,
	there are a $3$-dimensional quadratic Calabi-Yau Artin-Schelter regular algebra $S$ 
	and a twisting system $\theta$ such that $B \cong S^{\theta}$ (\cite[Theorem 4.4]{IM2}).
	Except for one case, a twisting system $\theta$ can be induced by a graded algebra automorphism of $S$.
	By using $M(E,\sigma)$,
	we can recover this fact in the case that $B$ is geometric (see Corollary \ref{thm.main}).

	\section{Preliminary}
	Throughout this paper, we fix an algebraically closed field $k$ of characteristic zero and
	assume that a graded algebra is an $\mathbb{N}$-graded algebra $A=\bigoplus_{i \in \mathbb{N}} A_i$ over $k$.
	A graded algebra $A=\bigoplus_{i \in \mathbb{N}} A_i$ is called {\it connected} if $A_0=k$. 
	Let $\GrAut_{k}\,A$ denote the group of graded algebra automorphisms of $A$.
	We denote by $\GrMod A$ the category of graded right $A$-modules.
	We say that two graded algebras $A$ and $A'$ are {\it graded Morita equivalent} if two categories $\GrMod A$ and $\GrMod A'$
	are equivalent.
	
	\subsection{Twisting systems and twisted algebras}
	\begin{defin}
		Let $A$ be a graded algebra. A set of graded $k$-linear automorphisms $\theta=\{ \theta_n \}_{n \in \mathbb{Z}}$
		of $A$ is called a twisting system on $A$ if
		$$
		\theta_n(a\theta_m(b))=\theta_n(a)\theta_{n+m}(b)
		$$
		for any $l,m,n \in \mathbb{Z}$ and $a \in A_m$, $b \in A_l$.
		The twisted algebra of $A$ by $\theta$, denoted by $A^{\theta}$, is a graded algebra $A$
		with a new multiplication $\ast$ defined by
		$$a \ast b =a\theta_m(b)$$
		for any $a \in A_m$, $b \in A_l$.
		A twisting system $\theta=\{ \theta_n \}_{n \in \mathbb{Z}}$ is called {\it algebraic}
		if $\theta_{m+n}=\theta_m \circ \theta_n$ for every $m,n \in \mathbb{Z}$.
	\end{defin}
    
    We denote by $\TS(A)$ the set of twisting systems on $A$.
    Zhang \cite{Z} found a necessary and sufficient algebraic condition for $\GrMod A \cong \GrMod A'$.
    
    \begin{thm}[{\cite[Theorem 3.5]{Z}}]
    	Let $A$ and $A'$ be graded algebras finitely generated in degree $1$ over $k$. Then
    	$\GrMod A \cong \GrMod A'$ if and only if $A'$ is isomorphic to a twisted algebra $A^{\theta}$
    	by a twisting system $\theta \in \TS(A)$.
    \end{thm}

    \begin{defin}
    	For a graded algebra $A$, we define
    	\begin{align*}
    		&\TSo(A):=\{ \theta \in \TS(A) \mid \theta_0={\rm id}_A \} \\
    		&{\rm TS}_{\rm alg}^{\mathbb{Z}}(A):=\{ \theta \in \TSo(A) \mid \theta \text{ is algebraic } \} \\
    		&{\rm Twist}(A):=\{ A^{\theta} \mid \theta \in \TS(A) \}/_{\cong}\\
    		&{\rm Twist}_{\rm alg}(A):=\{ A^{\theta} \mid \theta \in \TSa(A) \}/_{\cong}.
    	\end{align*}
    \end{defin}

    \begin{lem}[{\cite[Proposition 2.4]{Z}}]\label{lem.tso}
    	Let $A$ be a graded algebra. For every $\theta \in \TS(A)$,
    	there exists $\theta' \in \TSo(A)$ such that $A^{\theta} \cong A^{\theta'}$.
    \end{lem}

    It follows from Lemma \ref{lem.tso} that ${\rm Twist}(A)=\{ A^{\theta} \mid \theta \in \TSo(A) \}/_{\cong}$,
    so we may assume that $\theta \in \TSo(A)$ to study ${\rm Twist}(A)$.
    By the definition of twisting system,
    it follows that $\theta \in \TSa(A)$ if and only if $\theta_n=\theta_{1}^{n}$ for every $n \in \mathbb{Z}$ and
    $\theta_{1} \in \GrAut_{k}\,A$, so
    $${\rm Twist}_{\rm alg}(A)=\{ A^{\phi} \mid \phi \in \GrAut_{k}\,A \}/_{\cong}$$
    where $A^{\phi}$ means the twisted algebra of $A$ by $\{ \phi^{n} \}_{n \in \mathbb{Z}}$.
    
	\subsection{Geometric algebra}
	Let $V$ be a finite dimensional $k$-vector space and $A=T(V)/(R)$ be a quadratic algebra
	where $T(V)$ is a tensor algebra over $k$ and $R \subset V \otimes V$.
	Since an element of $R$ defines a multilinear function on $V^{\ast}\times V^{\ast}$,
	we can define a zero set associated to $R$ by
	$$\mathcal{V}(R)=\{
	(p,q) \in \mathbb{P}(V^{\ast}) \times \mathbb{P}(V^{\ast}) \mid g(p,q)=0 \text{ for any } g \in R
	\}.$$
	
	\begin{defin}\label{defin.geom}
		Let $A=T(V)/(R)$ be a quadratic algebra. A geometric pair $(E,\sigma)$ consists of a projective variety $E \subset \mathbb{P}(V^{\ast})$
		and $\sigma \in \Aut_{k} E$.
		\begin{enumerate}[{\rm (1)}]
			\item We say that $A$ satisfies (G1) if there exists a geometric pair $(E,\sigma)$ such that
			$$\mathcal{V}(R)=\{
			(p,\sigma(p)) \in \mathbb{P}(V^{\ast}) \times \mathbb{P}(V^{\ast}) \mid p \in E
			\}.$$
			In this case, we write $\mathcal{P}(A)=(E,\sigma)$, and call $E$ the {\it point variety} of $A$.
			
			\item We say that $A$ satisfies (G2) if there exists a geometric pair $(E,\sigma)$ such that
			$$R=\{
			g \in V \otimes V \mid g(p,\sigma(p))=0 \text{ for all } p \in E
			\}.$$
			In this case, we write $A=\mathcal{A}(E,\sigma)$.
			
			\item We say that $A$ is a geometric algebra if it satisfies both (G1) and (G2) with $A=\mathcal{A}(\mathcal{P}(A))$.
		\end{enumerate}
	\end{defin}

    For geometric algebras, Mori \cite{Mo} found a necessary and sufficient 
    geometric condition for $\GrMod A \cong \GrMod A'$. 
    
    \begin{thm}[{\cite[Theorem 4.7]{Mo}}]\label{thm.Mori}
    Let $A=\mathcal{A}(E,\sigma)$ and $A'=\mathcal{A}(E',\sigma')$ be geometric algebras.
    Then $\GrMod A \cong \GrMod A'$ if and only if there exists a sequence of automorphisms $\{ \tau_n \}_{n \in \mathbb{Z}}$
    of $\mathbb{P}(V^{\ast})$ for $n \in \mathbb{Z}$, each of which sends $E$ isomorphically onto $E'$, such that the diagram
    $$\xymatrix{
    	E\ar[r]^-{\tau_n|_E}\ar[d]_-{\sigma}&E'\ar[d]^-{\sigma'}\\
    	E\ar[r]_-{\tau_{n+1}|_E}&E'
    }$$
    commutes for every $n \in \mathbb{Z}$.
    \end{thm}
    
    \subsection{Artin-Schelter regular algebras and Calabi-Yau algebras}
    
    \begin{defin}
    	A connected graded algebra $A$ is called a {\it $d$-dimensional Artin-Schelter regular algebra} (simply {\it AS-regular algebra})
    	if $A$ satisfies the following conditions;
    	\begin{enumerate}[{\rm (1)}]
    		\item ${\rm gldim} A=d<\infty$,
    		\item ${\rm GKdim} A:=\inf\{ \al \in \mathbb{R} \mid \dim_{k}(\sum_{i=0}^{n}A_i) \leq n^{\al} \text{ for all } n\gg0 \} <\infty$, and,
    		\item $\Ext^{i}_{A}(k,A)=
    		\begin{cases}
    			k \,\,\,\text{ if } i=d, \\
    			0 \,\,\,\text{ if } i \neq d.
    		\end{cases}$
    	\end{enumerate}
    \end{defin}
    
    Artin and Schelter proved that a $3$-dimensional AS-regular algebra $A$ finitely generated in degree $1$ is
    isomorphic to one of the following forms:
    $$k\langle x,y,z \rangle/(f_1,f_2,f_3) \text{ or } k\langle x,y \rangle/(g_1,g_2)$$
    where $f_i$ are homogeneous polynomials of degree $2$ (the quadratic case) and
    $g_j$ are homogeneous polynomials of degree $3$ (the cubic case) (see \cite[Theorem 1.5]{AS}).
    
    We recall the definition of Calabi-Yau algebra introduced by \cite{G}.
    
    \begin{defin}[{\cite{G}}]
    	A $k$-algebra $S$ is called {\it $d$-dimensional Calabi-Yau}
    	if $S$ satisfies the following conditions:
    	\begin{enumerate}
    		\item ${\rm pd}_{S^e}\,S=d<\infty$, and
    		\item ${\rm Ext}^i_{S^e}(S,S^e) \cong
    		\begin{cases}
    			S \,\,\,\text{ if } i=d, \\
    			0 \,\,\,\text{ if } i \neq d
    		\end{cases}$ (as right $S^e$-modules)
    	\end{enumerate} 
        where $S^e=S^{\rm op} \otimes_{k} S$ is the enveloping algebra of $S$.
    \end{defin}
     
    The following theorem tells us that we may assume that $3$-dimensional quadratic AS-regular algebra
    is Calabi-Yau up to graded Morita equivalence.
    
    \begin{thm}[{\cite[Theorem 4.4]{IM2}}]
    	For every $3$-dimensional quadratic AS-regular algebra $A$,
    	there exists a $3$-dimensional quadratic Calabi-Yau AS-regular algebra $S$
    	such that $\GrMod\,A \cong \GrMod\,S$.
    \end{thm}
    
    \begin{lem}[{\cite[Lemma 2.8]{HMM}, \cite[Theorem 3.2]{IM}, \cite[Lemma 3.8]{MU}}]
    	\label{lem.SP}
    	Every $3$-dimensional geometric AS-regular algebra $A$ is graded Morita equivalent to
    	$S=k \langle x,y,z \rangle/(f_1,f_2,f_3)=\mathcal{A}(E,\sigma)$ in Table $1$.  
    \end{lem}
    
    \begin{center}
    	{\renewcommand\arraystretch{1.1}
    		{\small
    			\begin{tabular}{|c|p{4.8cm}|p{2.5cm}|p{5.5cm}|}
    				\multicolumn{4}{c}{Table $1$}
    				\\[5pt]
    				\hline
    				Type
    				& $f_{1}, f_{2}, f_{3}$
    				& $E$
    				& $\sigma$
    				\\ \hline\hline
    				$\rm{P}$ 
    				& 
    				$
    				\left\{
    				\begin{array}{ll}
    					yz-zy\\
    					zx-xz\\
    					xy-yx
    				\end{array}
    				\right.
    				$
    				& $\mathbb{P}^{2}$
    				&  \rule{0pt}{25pt}
    				$\sigma(a,b,c)=(a,b,c)$
    				\\[18pt] \hline
    				$\rm{S}$ 
    				& 
    				$\left\{
    				\begin{array}{ll}
    					yz-\alpha zy\\
    					zx-\alpha xz \quad\quad \alpha^{3}\neq 0,1\\
    					xy-\alpha yx
    				\end{array}
    				\right.
    				$
    				& $\begin{array}{ll}
    					& \mathcal{V}(x) \\
    					\cup & \mathcal{V}(y) \\
    					\cup & \mathcal{V}(z)
    				\end{array}$
    				& \rule{0pt}{25pt}
    				$\left\{
    				\begin{array}{ll}
    					\sigma(0,b,c)=(0,b,\alpha c)\\
    					\sigma(a,0,c)=(\alpha a,0,c)\\
    					\sigma(a,b,0)=(a,\alpha b,0)
    				\end{array}
    				\right.
    				$
    				\\[18pt] \hline
    				$\rm{S'}$ 
    				&
    				$\left\{
    				\begin{array}{ll}
    					yz-\alpha zy+x^{2}\\
    					zx-\alpha xz\quad\quad \alpha^{3}\neq 0,1\\
    					xy-\alpha yx
    				\end{array}
    				\right.
    				$
    				& $\begin{array}{ll}
    					& \mathcal{V}(x) \\
    					\cup & \mathcal{V}(x^2-\la yz) \\
    					& \la=\frac{\al^3-1}{\al}
    				\end{array}$
    				& \rule{0pt}{25pt}
    				$\left\{
    				\begin{array}{ll}
    					\sigma(0,b,c)=(0,b, \alpha c)\\
    					\sigma(a,b,c)=(a,\alpha b,\alpha^{-1} c)\\
    				\end{array}
    				\right.
    				$
    				\\[18pt] \hline
    				$\rm{T}$ 
    				& 
    				$\left\{
    				\begin{array}{ll}
    					yz-zy+x^{2}\\
    					zx-xz+y^{2}\\
    					xy-yx
    				\end{array}
    				\right.
    				$
    				& $\begin{array}{ll}
    					& \mathcal{V}(x+y) \\
    					\cup & \mathcal{V}(\ep x+y) \\
    					\cup & \mathcal{V}(\ep^2x+y)
    				\end{array}$
    			    
    			    $\ep^3=1$, $\ep,\ep^2 \neq 1$
    				& 
    				$\left\{
    				\begin{array}{ll}
    					\sigma(a,-a,c)=(a,- a, a+c)\\
    					\sigma(a,-\ep a,c)=(a,-\ep a, \ep^2a+c)\\
    					\sigma(a,-\ep^2 a,c)=(a,-\ep^2a, \ep a+c)
    				\end{array}
    				\right.
    				$
    				\\ \hline
    				$\rm{T'}$ 
    				& 
    				$\left\{
    				\begin{array}{ll}
    					yz-zy+xy+yx\\
    					zx-xz+x^{2}-yz-zy+y^{2}\\
    					xy-yx-y^{2}
    				\end{array}
    				\right.
    				$
    				& $\begin{array}{ll}
    					& \mathcal{V}(x) \\
    					\cup & \mathcal{V}(y^2-xz)
    				\end{array}$
    				& 
    				$\left\{
    				\begin{array}{ll}
    					\sigma(a,0,c)=(0,b,b+c), \\
    					\sigma(a,b,c)=(a,-a+b,a-2b+c) 
    				\end{array}
    				\right.
    				$
    				\\ \hline
    				$\rm{NC}$ 
    				& 
    				$\left\{
    				\begin{array}{ll}
    					yz-\alpha zy+x^{2}\\
    					zx-\alpha xz+y^{2}\quad\quad \alpha^{3}\neq 0,1\\
    					xy-\alpha yx
    				\end{array}
    				\right.
    				$
    				& $\begin{array}{ll}
    					\mathcal{V}(x^{3}+y^{3} \\
    					\qquad -\la xyz) \\
    					\la=\frac{\al^3-1}{\al}\end{array}$
    				& \rule{0pt}{25pt}
    				$
    				\begin{array}{ll}
    					\sigma(a,b,c)=\\
    					(a,\alpha b,
    					-\frac{a^{2}}{b}+\alpha^{2}c) \\
    				\end{array}
    				$
    				\\[18pt] \hline
    				$\rm{CC}$ 
    				& 
    				$\left\{
    				\begin{array}{ll}
    					yz-zy+y^{2}+3x^{2}\\
    					zx-xz+yx+xy-yz-zy\\
    					xy-yx-y^{2}
    				\end{array}
    				\right.
    				$
    				& $\mathcal{V}(x^{3}-y^{2}z)$
    				& $
    				\begin{array}{ll}
    					\sigma(a,b,c)=\\
    					(a-b,b,
    					-3\frac{a^2}{b}+3a-b+c) 
    				\end{array}
    				$
    				\\ \hline
    				EC 
    				& 
    				$\left\{
    				\begin{array}{ll}
    					\al yz+\be zy+\ga x^{2} \\
    					\al zx+\be xz+\ga y^{2} \\
    					\al xy+\be yx+\ga z^{2}
    				\end{array}
    				\right.
    				$
    				
    				$(\al^3+\be^3+\ga^3)^3\neq (3\al\be\ga)^3$, $\al\be\ga\neq 0$ 
    				& $\begin{array}{ll}
    					\mathcal{V}(x^{3}+y^{3}+z^{3} \\
    					\qquad -\lambda xyz), \\  
    					\lambda=\frac{\al^3+\be^3+\ga^3}{\al\be\ga}
    				\end{array}$
    				& 
    				\rule{0pt}{35pt}
    				$\sigma_p$ where $p=(\al,\be,\ga)\in E$
    				\\[30pt] \hline
    			\end{tabular}
    	}}
    \end{center}
    
\begin{rem}
	The original definition of geometric algebra given by Mori \cite{Mo} is different from our definition.
	In the sense of Definition \ref{defin.geom}, there exists a
	$3$-dimensional quadratic AS-regular algebra which is not a geometric algebra.
	Strictly speaking, a $3$-dimensional quadratic AS-regular algebra is a geometric algebra in our sense
	if and only if the \textquotedblleft point scheme\textquotedblright is reduced.
\end{rem}   

\section{Twisted algebras of geometric algebras}
In this section, we study twisted algebras of geometric algebras.
Let $E \subset \mathbb{P}(V^{\ast})$ be a projective variety
where $V$ is a finite dimensional $k$-vector space.
We use the following notations
introduced in \cite{CG}:

\begin{defin}
	Let $E \subset \mathbb{P}(V^{\ast})$ be a projective variety and $\sigma \in \Aut_{k}\,E$.
	We define
	\begin{align*}
		&\Aut_{k}(E \uparrow \mathbb{P}(V^{\ast})):=
		\{ \tau \in \Aut_{k}\,E \mid \tau=\overline{\tau}|_{E}
		\text{ for some } \overline{\tau} \in \Aut_{k}\,\mathbb{P}(V^{\ast}) \}, \\
		&\Aut_{k}(\mathbb{P}(V^{\ast}) \downarrow E):=
		\{ \tau  \in \Aut_{k}\,\mathbb{P}(V^{\ast}) \mid \tau|_{E} \in \Aut_{k}\,E  \}, \\
		&\Z:=\{ \tau \in \Aut_{k}(\mathbb{P}(V^{\ast}) \downarrow E) \mid \sigma\tau|_{E}\sigma^{-1}=\tau|_{E}\}, \\
		&\M:=\{ \tau \in \Aut_{k}(\mathbb{P}(V^{\ast}) \downarrow E) \mid (\tau|_{E}\sigma)^{i}\sigma^{-i} \in \Aut_{k}(E \uparrow \mathbb{P}(V^{\ast}))\,\,\, \forall i\in\mathbb{Z} \}, \\
		&\N:=\{ \tau \in \Aut_{k}(\mathbb{P}(V^{\ast}) \downarrow E) \mid \sigma\tau|_{E}\sigma^{-1} \in \Aut_{k}(E \uparrow \mathbb{P}(V^{\ast})) \}.
	\end{align*}
\end{defin}
Note that $ \Z \subset \M \subset \N \subset \Aut_{k}(\mathbb{P}(V^{\ast}) \downarrow E)$, and
$\Z$, $\N$ are subgroups of $\Aut_{k}(\mathbb{P}(V^{\ast}) \downarrow E)$.

\begin{lem}\label{lem.spps}
	Let $E \subset \mathbb{P}(V^{\ast})$ be a projective variety and $\sigma \in \Aut_{k}\,E$.
	If $\sigma\Aut_{k}(E \uparrow \mathbb{P}(V^{\ast}))=\Aut_{k}(E \uparrow \mathbb{P}(V^{\ast}))\sigma$,
	then $M(E,\sigma)=N(E,\sigma)=\Aut_{k}(\mathbb{P}(V^{\ast}) \downarrow E)$.
\end{lem}

\begin{proof}
	Since $M(E,\sigma) \subset N(E,\sigma) \subset \Aut_{k}(\mathbb{P}(V^{\ast}) \downarrow E)$ in general,
	it is enough to show that $\Aut_{k}(\mathbb{P}(V^{\ast}) \downarrow E) \subset M(E,\sigma)$.
	We will show that for any $\tau \in \Aut_{k}(\mathbb{P}(V^{\ast}) \downarrow E)$,
	$(\tau|_E\sigma)^{i}\sigma^{-i} \in
	\Aut_{k}(E \uparrow \mathbb{P}(V^{\ast}))$
	for every $i \in \mathbb{Z}$ by induction so that $\tau \in M(E,\sigma)$.
	The claim is trivial for $i=0$.
	If $(\tau|_E\sigma)^{i}\sigma^{-i} \in \Aut_{k}(E \uparrow \mathbb{P}(V^{\ast}))$ for some $i \ge 0$,
	then $(\tau|_E\sigma)^{i+1}\sigma^{-i-1}=
	\tau|_E\sigma((\tau|_E\sigma)^i\sigma^{-i})\sigma^{-1} \in \Aut_{k}(E \uparrow \mathbb{P}(V^{\ast}))$.
	If $(\tau|_E\sigma)^{-i}\sigma^i \in \Aut_{k}(E \uparrow \mathbb{P}(V^{\ast}))$ for some $i \ge 0$,
	then $$(\tau|_E\sigma)^{-(i+1)}\sigma^{i+1}=\sigma^{-1}\tau|_E^{-1}((\tau|_E\sigma)^{-i}\sigma^i)\sigma
	\in \Aut_{k}(E \uparrow \mathbb{P}(V^{\ast})).$$
\end{proof}

Let $A=\mathcal{A}(E,\sigma)$ be a geometric algebra.
The map $\Phi : \TSo(A) \to \Aut_{k}\,\mathbb{P}(A_1^{\ast})$ is defined by
$\Phi(\theta):=\overline{(\theta_{1}|_{A_{1}})^{\ast}}$.

\begin{lem}\label{Mtwist}
	Let $A=\mathcal{A}(E,\sigma)$ be a geometric algebra. 
	Then $$\Phi(\TSo(A))=M(E,\sigma).$$
\end{lem}

\begin{proof}
	Let $\theta \in \TSo(A)$. We set $V:=A_1=(A^{\theta})_1$.
	Then $\theta_n$ is also a graded $k$-linear isomorphism from $A^{\theta}$ to $A$ and
	satisfies $\theta_n(a \ast b)=\theta_n(a)\theta_{n+m}(b)$
	for every $n, m, l\in \mathbb{Z}$ and $a \in A^{\theta}_m$, $b \in A^{\theta}_l$.
	Let $\tau_n : \mathbb{P}(V^{\ast}) \to \mathbb{P}(V^{\ast})$ be automorphisms
	induced by the duals of $\theta_n|_V : V \to V$.
	By \cite[Remark 15]{CG},
	$\tau_n \in \Aut_{k}(\mathbb{P}(V^{\ast}) \downarrow E)$ and 
	the diagram of automorphisms
	$$\xymatrix{
		E\ar[r]^-{\tau_n|_E}\ar[d]_-{\sigma}&E\ar[d]^-{\tau_{1}|_E\sigma}\\
		E\ar[r]_-{\tau_{n+1}|_E}&E
	}$$
	commutes for every $n \in \mathbb{Z}$.
	Then $(\tau_{1}|_E\sigma)^{n}\sigma^{-n}=\tau_n|_E \in \Aut_{k}(E \uparrow \mathbb{P}(V^{\ast}))$
	for every $n \in \mathbb{Z}$, so
	it holds that $\Phi(\theta)=\tau_{1} \in \M$.
	
	Conversely, let $\tau \in M(E,\sigma)$.
	Since $(\tau|_E\sigma)^{n}\sigma^{-n} \in \Aut_{k}(E \uparrow \mathbb{P}(V^{\ast}))$,
	there is an automorphism $\tau_n \in \Aut_{k}\,\mathbb{P}(V^{\ast})$ such that
	$\tau_n|_E=(\tau|_E\sigma)^{n}\sigma^{-n}$ for every $n \in \mathbb{Z}$.
	By \cite[Remark 15]{CG}, there exists $\theta \in \TSo(A)$ such that
	$\overline{(\theta_{n}|_{A_{1}})^{\ast}}=\tau_n$ for every $n \in \mathbb{Z}$.
	Hence, it follows that
	$\Phi(\theta)=\overline{(\theta_{1}|_{A_{1}})^{\ast}}=\tau$.
\end{proof}

Let $A=T(V)/I$ be a connected graded algebra.
Let $\Psi: \GrAut_{k}\,A \to \PGl(V)$ be a group homomorphism defined by $\Psi(\phi)=\overline{\phi|_V}$.
We set $${\rm PGrAut}_{k}\,A:=\GrAut_{k}\,A/\Ker \Psi.$$

\begin{lem}\label{Ztwist}
	Let $A=\mathcal{A}(E,\sigma)$ be a geometric algebra.
	\begin{enumerate}[{\rm (1)}]
		\item $\Phi(\TSa(A))=Z(E,\sigma)$.
		\item ${\rm PGrAut}_{k}\,A \cong Z(E,\sigma)^{\rm op}.$
	\end{enumerate}
\end{lem}

\begin{proof}
	(1) Let $\theta=\{ \theta_{1}^{n} \}_{n \in \mathbb{Z}} \in \TSa(A)$.
	We set $V:=A_1=(A^{\theta})_1$.
	Let $\tau_n : \mathbb{P}(V^{\ast}) \to \mathbb{P}(V^{\ast})$ be automorphisms
	induced by the duals of $\theta_{1}^n|_V : V \to V$.
	Then we can write $\tau_n=\tau_{1}^n$ for every $n \in \mathbb{Z}$.
	By the proof of Lemma \ref{Mtwist}, it follows that $(\tau_{1}|_E)^n=(\tau_1|_E\sigma)^n\sigma^{-n}$
	for every $n \in \mathbb{Z}$. If $n=2$, then $\tau_{1}|_E\sigma=\sigma\tau_{1}|_E$,
	so $\Phi(\theta)=\tau_{1} \in Z(E,\sigma)$.

	Conversely, let $\tau \in Z(E,\sigma)$.
	Since $\tau|_E\sigma=\sigma\tau|_E$, $(\tau|_E\sigma)^{n}\sigma^{-n}=(\tau|_E)^{n}$
	for every $n \in \mathbb{Z}$.
	By \cite[Remark 15]{CG}, there exists $\theta=\{ \theta_{1}^n \}_{n \in \mathbb{Z}} \in \TSa(A)$
	such that $\overline{(\theta_{1}|_{V})^{\ast}}^n=\tau^n$ for every $n \in \mathbb{Z}$.
	Hence, it follows that $\Phi(\theta)=\tau$.

	(2) By the following commutative diagram;
	$$\xymatrix{
	& \TSo(A) \ar@{}[d]|{\bigcup} \ar[rd]^-{\Phi} & \\
	\GrAut_{k}\,A \ar[r]^-{\text{ bij }} \ar[rd]_-{\Psi} & \TSa(A) \ar[r] & \Aut_{k} \mathbb{P}(V^{\ast})^{\rm op} \ar@{=}[d] \\
	& \PGl(V) \ar[r]^-{\cong} & \PGl(V^{\ast})^{\rm op} \\
    }
	$$
	it follows that ${\rm PGrAut}_{k}\,A \cong \Phi(\TSa(A))=Z(E,\sigma)^{\rm op}$.
\end{proof}

\begin{thm}\label{thm.twist}
	Let $A=\mathcal{A}(E,\sigma)$ be a geometric algebra.
	\begin{enumerate}[{\rm (1)}]
		\item ${\rm Twist}(A)=\{ \mathcal{A}(E,\tau|_E\sigma) \mid \tau \in \M \}/_{\cong}$.
		\item ${\rm Twist}_{\rm alg}(A)=\{ \mathcal{A}(E,\tau|_E\sigma) \mid \tau \in \Z \}/_{\cong}$.
	\end{enumerate}
\end{thm}

\begin{proof}
	By \cite[Proposition 13]{CG}, for every $\theta \in \TSo(A)$,
	$A^{\theta} \cong \mathcal{A}(E,\Phi(\theta)|_E\sigma)$.
	By Lemma \ref{Mtwist}, it follows that
	\begin{align*}
		{\rm Twist}(A)&:=\{ A^{\theta} \mid \theta \in \TSo(A) \}/_{\cong} \\
		&=\{ \mathcal{A}(E,\Phi(\theta)|_E\sigma) \mid \theta \in \TSo(A) \}/_{\cong} \\
		&=\{ \mathcal{A}(E,\tau|_E\sigma) \mid \tau \in \Phi(\TSo(A)) \}/_{\cong} \\
		&=\{ \mathcal{A}(E,\tau|_E\sigma) \mid \tau \in M(E,\sigma) \}/_{\cong}.
	\end{align*}
	By Lemma \ref{Ztwist}, we can similarly show that
	$${\rm Twist}_{\rm alg}(A)=
	\{ \mathcal{A}(E,\tau|_E\sigma) \mid \tau \in Z(E,\sigma) \}/_{\cong}.$$
\end{proof}

\section{Twisted algebras of $3$-dimensional geometric AS-regular algebras}
In this section, we classify twisted algebras of $3$-dimensional geometric AS-regular algebras.
We recall that for connected graded algebras $A$ and $A'$ generated in degree 1,
$\GrMod\,A \cong \GrMod\,A'$ if and only if $A' \in {\rm Twist}(A)$, so
$$
{\rm Twist}(A)=\{ A' \mid \GrMod\,A' \cong \GrMod\,A \}/_{\cong}.
$$
By Lemma \ref{lem.SP}, we may assume that $A$ is a $3$-dimensional geometric Calabi-Yau AS-regular algebra
in Table $1$
to compute ${\rm Twist}(A)$.
The algebras in Table 1 are called {\it standard} in this paper.
For $3$-dimensional standard AS-regular algebras,
we will compute the subsets
$\Z$ and $\M$ of $\Aut_{k}(\mathbb{P}^2 \downarrow E)$.
We remark that some of the computations were given in \cite[Section 4]{CG}.

For a $3$-dimensional geometric AS-regular algebra $\mathcal{A}(E,\sigma)$,
the map $$\ddPAut \to \uuPAut; \tau \mapsto \tau|_E$$ is a bijection, so
we identify $\tau \in \Aut_{k}(\mathbb{P}^2 \downarrow E)$ with $\tau|_E \in \Aut_{k}(E \uparrow \mathbb{P}^2)$
if there is no potential confusion.

Let $E$ be an elliptic curve in $\mathbb{P}^2$. We use a {\it Hesse form}
$$E=\mathcal{V}(x^3+y^3+z^3-3 \la xyz)$$ where $\la \in k$ with $\la^3 \neq 1$.
It is known that every elliptic curve in $\mathbb{P}^2$ can be written in this form (see \cite[Corollary 2.18]{F}).
The $j$-invariant of a Hesse form $E$ is given by
$j(E)=\frac{27\la^3(\la^3+8)^3}{(\la^3-1)}$
(see \cite[Proposition 2.16]{F}).
The $j$-invariant $j(E)$ classifies elliptic curves in $\mathbb{P}^2$ up to projective equivalence
(see \cite[Theorem IV 4.1 (b)]{H}).
We fix the group structure on $E$ with the zero element $o:=(1,-1,0) \in E$ (see \cite[Theorem 2.11]{F}).
For a point $p \in E$, a {\it translation} by $p$, denoted by $\sigma_{p}$, is an automorphism of $E$
defined by $\sigma_{p}(q)=p+q$ for every $q \in E$.
We define $\Aut_{k}(E,o):=\{ \sigma \in \Aut_{k}\,E \mid \sigma(o)=o \}$.
It is known that $\Aut_{k}(E,o)$ is a finite cyclic subgroup of $\Aut_{k}\,E$ (see \cite[Corollary IV 4.7]{H}).

\begin{lem}[{\cite[Theorem 4.6]{IM}}]\label{lem.tau}
	A generator of $\Aut_{k}(E,o)$ is given by
	\begin{enumerate}
		\item $\tau_E(a,b,c):=(b,a,c)$ if $j(E) \neq 0,12^3$,
		\item $\tau_E(a,b,c):=(b,a,\ep c)$ if $\la=0$ (so that $j(E)=0$),
		\item $\tau_E(a,b,c):=(\ep^2a+\ep b+c,\ep a+\ep^2b+c,a+b+c)$ if $\la=1+\sqrt{3}$ (so that $j(E)=12^3$)
	\end{enumerate}
    where $\ep$ is a primitive $3$rd root of unity.
    In particular, $\Aut_{k}(E,o)$ is a subgroup of $\uuPAut=\ddPAut$.
\end{lem}

\begin{rem}
	When $j(E)=0, 12^3$, we may fix $\la=0,1+\sqrt{3}$ respectively as in Lemma \ref{lem.tau},
	because if two elliptic curves $E$ and $E'$ in $\mathbb{P}^2$ are projectively equivalent,
	then for every $\mathcal{A}(E,\sigma)$, there exists an automorphism $\sigma' \in \Aut_{k}\,E'$
	such that $\mathcal{A}(E,\sigma) \cong \mathcal{A}(E',\sigma')$ (see \cite[Lemma 2.6]{MU}).
\end{rem}

It follows from \cite[Proposition 4.5]{IM} that every automorphism $\sigma \in \Aut_{k}\,E$
can be written as $\sigma=\sigma_{p}\tau_E^i$ where $\sigma_{p}$ is a translation by a point $p \in E$,
$\tau_E$ is a generator of $\Aut_{k}(E,o)$ and $i \in \mathbb{Z}_{|\tau_E|}$.
For any $n \ge 1$, we call a point $p \in E$ {\it $n$-torsion}
if $np=o$.
We set $E[n]:=\{ p \in E \mid np=o \}$ and $T[n]:=\{ \sigma_{p} \in \Aut_{k}\,E \mid p \in E[n] \}$.

If $A=\mathcal{A}(E,\sigma)$ is a $3$-dimensional standard AS-regular algebra,
we write $$\ddPAut=Z(E) \rtimes G(E)$$ as in Table 2 below
where $Z(E) \leq \ddPAut$ with $Z(E) \subset Z(E,\sigma)$ and
$G(E) \leq \ddPAut$ so that
$$N(E,\sigma)=Z(E) \rtimes (G(E) \cap N(E,\sigma)).$$
\begin{center}
	{\renewcommand\arraystretch{1.1}
		{\small
			\begin{tabular}{|c|c|c|}
				\multicolumn{3}{c}{{\rm Table 2}}
				\\[5pt]
				\hline
				{\rm Type}
				& $Z(E)$
				& $G(E)$
				\\ \hline \hline
				{\rm P}
				& $\PGl_{3}(k)$
				& $\{ \rm id \}$
				\\ \hline
				{\rm S}
				& $\left\{
				\begin{pmatrix}
					1 & 0 & 0 \\
					0 & e & 0 \\
					0 & 0 & i
				\end{pmatrix} \middle| ei \neq 0
				\right\}
				\rtimes \left\langle
				\begin{pmatrix}
					0 & 1 & 0 \\
					0 & 0 & 1 \\
					1 & 0 & 0
				\end{pmatrix} \right\rangle$
			    &
				$\left\langle
				\begin{pmatrix}
					0 & 1 & 0 \\
					1 & 0 & 0 \\
					0 & 0 & 1
				\end{pmatrix}
				\right\rangle$
				\\ \hline
				{\rm S'}
				& $\left\{
				\begin{pmatrix}
					1 & 0 & 0 \\
					0 & e & 0 \\
					0 & 0 & e^{-1}
				\end{pmatrix} \middle|\, e \neq 0
				\right\}$
				&
				$\left\langle
				\begin{pmatrix}
					1 & 0 & 0 \\
					0 & 0 & 1 \\
					0 & 1 & 0
				\end{pmatrix} \right\rangle$
				\\ \hline
				{\rm T}
				& $\left\{
				\begin{pmatrix}
					1 & 0 & 0 \\
					0 & e & 0 \\
					g & h & e^2
				\end{pmatrix} \middle| e^3=1 \right\} \rtimes \left\langle
			    \begin{pmatrix}
			    	0 & 1 & 0 \\
			    	1 & 0 & 0 \\
			    	0 & 0 & -1
			    \end{pmatrix} \right\rangle$
				& $\left\{
				\begin{pmatrix}
					1 & 0 & 0 \\
					0 & 1 & 0 \\
					0 & 0 & i
				\end{pmatrix} \middle| i \neq 0 \right\}$
				\\ \hline
				{\rm T'}
				& $
				\left\{
				\begin{pmatrix}
					1 & 0 & 0 \\
					d & 1 & 0 \\
					d^2 & 2d & 1
				\end{pmatrix}
				\right\}$
				& $\left\{
				\begin{pmatrix}
					1 & 0 & 0 \\
					0 & e & 0 \\
					0 & 0 & e^2
				\end{pmatrix} \middle| e \neq 0
				\right\}$
				\\ \hline
				{\rm NC}
				& $\left\langle
				\begin{pmatrix}
					1 & 0 & 0 \\
					0 & \ep & 0 \\
					0 & 0 & \ep^2
				\end{pmatrix}\right\rangle$
			    & $\left\langle
				\begin{pmatrix}
					0 & 1 & 0 \\
					1 & 0 & 0 \\
					0 & 0 & 1 
				\end{pmatrix}\right\rangle$
				\\ \hline
				{\rm CC}
				& $\{ \rm id \}$
				& $\left\{
				\begin{pmatrix}
					1 & 0 & 0 \\
					0 & e & 0 \\
					0 & 0 & e^2
				\end{pmatrix} \middle| e \neq 0
				\right\}$
				\\ \hline
				{\rm EC}
				& $T[3]$
				& $\Aut_{k}(E,o)$
				\\ \hline
			\end{tabular}
	}}
\end{center}

	The Table 2 above can be checked by the following three steps:
	\begin{itemize}
		\item Step $1$: Calculate $\ddPAut$.
		\item Step $2$: Find $Z(E) \leq \ddPAut$ with $Z(E) \subset Z(E,\sigma) \cong ({\rm PGrAut}_{k}\,A)^{\rm op}$
		(see Lemma \ref{Ztwist} (2)).
		\item Step $3$: Find $G(E) \leq \ddPAut$.
	\end{itemize}
	$\ddPAut$ and ${\rm PGrAut}_{k}\,A$ were computed in \cite{Y}.
	We explain these steps for Type S.
	By Lemma \ref{lem.SP},
	$E=\mathcal{V}(x) \cup \mathcal{V}(y) \cup \mathcal{V}(z)$ and
	$$\begin{cases}
		\sigma(0,b,c)=(0,b,\al c) \\
		\sigma(a,0,c)=(\al a,0,c) \\
		\sigma(a,b,0)=(a,\al b,0)
	\end{cases}$$
    where $\al^3 \neq 0,1$.
	By \cite[Lemma 3.2.1]{Y},
	$$\ddPAut=\left\{
	\begin{pmatrix}
	1 & 0 & 0 \\
    0 & e & 0 \\
	0 & 0 & i
	\end{pmatrix} \middle| ei \neq 0
	\right\}
	\rtimes \left\langle
	\begin{pmatrix}
	0 & 1 & 0 \\
	0 & 0 & 1 \\
	1 & 0 & 0
	\end{pmatrix},
	\begin{pmatrix}
	0 & 1 & 0 \\
	1 & 0 & 0 \\
	0 & 0 & 1
	\end{pmatrix}
	\right\rangle.$$
	By \cite[Theorem 3.3.1]{Y}
	\begin{align*}
	Z(E,\sigma)&=({\rm PGrAut}_{k}\,A)^{\rm op} \\
	&=
	\begin{cases}
		\left\{
		\begin{pmatrix}
			1 & 0 & 0 \\
			0 & e & 0 \\
			0 & 0 & i
		\end{pmatrix} \middle| ei \neq 0
		\right\} \rtimes \left\langle
		\begin{pmatrix}
			0 & 1 & 0 \\
			0 & 0 & 1 \\
			1 & 0 & 0
		\end{pmatrix}
		\right\rangle &\text{if } \sigma^2 \neq {\rm id}, \\
		\ddPAut &\text{if } \sigma^2={\rm id},
	\end{cases}
    \end{align*}
	so we may take
	$Z(E)=\left\{
	\begin{pmatrix}
		1 & 0 & 0 \\
		0 & e & 0 \\
		0 & 0 & i
	\end{pmatrix} \middle| ei \neq 0
	\right\} \rtimes \left\langle
	\begin{pmatrix}
		0 & 1 & 0 \\
		0 & 0 & 1 \\
		1 & 0 & 0
	\end{pmatrix}
	\right\rangle$
	and $G(E)=\left\langle
	\begin{pmatrix}
		0 & 1 & 0 \\
		1 & 0 & 0 \\
		0 & 0 & 1
	\end{pmatrix}
	\right\rangle$.
	
		\begin{rem}
		By Table 2,
		\begin{enumerate}[{\rm (1)}]
			\item $|G(E)|<\infty$ if and only if $A$ is of Type P, S, S', NC, EC, and, in this case,
			there exists $\tau_E \in \ddPAut$ such that $G(E)=\langle \tau_E \rangle$ is a finite cyclic group.
			
			\item $|G(E)|<\infty$ but $|G(E)| \neq 2$ if and only if $A$ is of Type P
			($|G(E)|=1$), or Type EC with $j(E)=0$ ($|G(E)=6|$), or Type EC with $j(E)=12^3$ ($|G(E)|=4$).
		\end{enumerate}
	\end{rem}

\begin{thm}\label{thm.ctt}
	If $A=\mathcal{A}(E,\sigma)$ is a $3$-dimensional quadratic AS-regular algebra of Type T, T', CC
	(so that $|\sigma|=\infty$ (cf. \cite{IMo})), then $Z(E,\sigma)=M(E,\sigma)=N(E,\sigma)$.
\end{thm}

\begin{proof}
	Writing $\ddPAut=Z(E) \rtimes G(E)$ as in Table 2,
	it is enough to show that $G(E) \cap N(E,\sigma)=\{ {\rm id} \}$.
	
	Type T: For every $\tau=\begin{pmatrix}
		1 & 0 & 0 \\
		0 & 1 & 0 \\
		0 & 0 & i
	\end{pmatrix} \in G(E)$,
    $$\begin{cases}
    	\sigma\tau|_E\sigma^{-1}(a,-a,c)=(a,-a,(1-i)a+ic) \\
    	\sigma\tau|_E\sigma^{-1}(a,-\ep a,c)=(a,-\ep a,\ep^2(1-i)a+ic) \\
    	\sigma\tau|_E\sigma^{-1}(a,-\ep^2 a,c)=(a,-\ep^2a,\ep(1-i)a+ic).
    \end{cases}$$
    If $\tau \in N(E,\sigma)$, then
    there exists $\overline{\tau} \in \Aut_{k}\,\mathbb{P}^2$ such that $\sigma\tau|_E\sigma^{-1}=\overline{\tau}|_E$.
    Then $1-i=\ep^2(1-i)=\ep (1-i)$, so $i=1$.
	
	Type T$'$: For every $\tau=
	\begin{pmatrix}
		1 & 0 & 0 \\
		0 & e & 0 \\
		0 & 0 & e^2
	\end{pmatrix}
	\in G(E)$,
	$$\begin{cases}
		\sigma\tau|_E\sigma^{-1}(0,b,c)=(0,eb,e(1-e)b+e^2c) \\
		\sigma\tau|_E\sigma^{-1}(a,b,c)=(a,(e-1)a+eb,(e-1)^2a+2e(e-1)b+e^2c).
	\end{cases}$$
    If $\tau \in N(E,\sigma)$, then
    there exists $\overline{\tau} \in \Aut_{k}\,\mathbb{P}^2$ such that $\sigma\tau|_E\sigma^{-1}=\overline{\tau}|_E$.
    Then $e(1-e)=2e(e-1)$, so $e=1$.
	
	Type CC: For every $\tau=
	\begin{pmatrix}
		1 & 0 & 0 \\
		0 & e & 0 \\
		0 & 0 & e^2
	\end{pmatrix} \in G(E)=\ddPAut$,
    \begin{align*}
    &\sigma\tau|_E\sigma^{-1}(a,b,c) \\
    &=
    \left(
    a+(1-e)b,eb,-3\frac{(a+b)^2}{eb}+3(a+b)-eb+e^{-2}\left( 3\frac{a^2}{b}+3a+b+c \right)
    \right).
    \end{align*}
    If $\tau \in N(E,\sigma)$, then
    there exists $\overline{\tau} \in \Aut_{k}\,\mathbb{P}^2$ such that $\sigma\tau|_E\sigma^{-1}=\overline{\tau}|_E$.
    Then there exists $0 \neq e' \in k$ such that
    $$(1+(1-e)b,eb)=(1,e'b) \text{ in } \mathbb{P}^1$$
    for every $(1,b,c) \in E$, so $e'=e=1$.
\end{proof}
\begin{lem}\label{lem.ztau}
	Let $A=\mathcal{A}(E,\sigma)$ be a $3$-dimensional standard AS-regular algebra of Type S, S', NC or EC.
	For every $i \ge 1$, $\sigma^{i} \in \uuPAut=\ddPAut$
	if and only if $\sigma^{3i}={\rm id}$.
\end{lem}

\begin{proof}
	This Lemma follows from \cite[Theorem 3.4]{IMo}.
\end{proof}
\begin{lem}\label{lem.gnmn}
	Let $A=\mathcal{A}(E,\sigma)$ be a $3$-dimensional standard AS-regular algebra.
	If $\sigma\tau\sigma^{-1}, \sigma^{-1}\tau\sigma \in \uuPAut$ for every $\tau \in G(E)$,
	then $M(E,\sigma)=N(E,\sigma)=\ddPAut$.
\end{lem}

\begin{proof}
	Every $\tau \in \ddPAut$ can be written as $\tau=\tau_1\tau_2$ where $\tau_1 \in Z(E)$, $\tau_2 \in G(E)$.
	Since $\sigma\tau\sigma^{-1}=\sigma\tau_1\tau_2\sigma^{-1}=\tau_1\sigma\tau_2\sigma^{-1}$,
	it holds that $\sigma\tau\sigma^{-1} \in \uuPAut$.
	Similarly, every $\tau \in \ddPAut$ can be written as $\tau=\tau_2\tau_1$ where $\tau_1 \in Z(E)$,
	$\tau_2 \in G(E)$, so $\sigma^{-1}\tau\sigma=\sigma^{-1}\tau_2\tau_1\sigma=
	\sigma^{-1}\tau_2\sigma\tau_{1}$, hence $\sigma^{-1}\tau\sigma \in \uuPAut$.
	The result now follows from Lemma \ref{lem.spps}.
\end{proof}

\begin{thm}\label{thm.62zm}
	Let $A=\mathcal{A}(E,\sigma)$ be a $3$-dimensional standard AS-regular algebra
	such that $|G(E)|=2$.
	\begin{enumerate}[{\rm (1)}]
		\item If $\sigma^2={\rm id}$, then $Z(E,\sigma)=M(E,\sigma)=N(E,\sigma)=\ddPAut$.
		\item If $\sigma^{6}={\rm id}$, then $M(E,\sigma)=N(E,\sigma)=\ddPAut$.
		\item If $\sigma^{6} \neq {\rm id}$, then $Z(E)=Z(E,\sigma)=M(E,\sigma)=N(E,\sigma)$.
	\end{enumerate}
\end{thm}

\begin{proof}
	(1) We will give a proof for Type S'.
	The other types are proved similarly.
	By Lemma \ref{lem.SP} and Table 2, $E=\mathcal{V}(x) \cup \mathcal{V}(x^2-\la yz)$,
	$$\begin{cases}
		\sigma(0,b,c)=(0,b,\al c) \\
		\sigma(a,b,c)=(a,\al b,\al^{-1}c)
	\end{cases}$$
    where $\la=\frac{\al^3-1}{\al}$ and $\al^3 \neq 0,1$, and
    $\tau_E=
    \begin{pmatrix}
    	1 & 0 & 0 \\
    	0 & 0 & 1 \\
    	0 & 1 & 0 
    \end{pmatrix}$.
	In general, $Z(E) \subset Z(E,\sigma) \subset M(E,\sigma) \subset N(E,\sigma) \subset \ddPAut$.
	In this case, $\sigma^2={\rm id}$ if and only if $\al^{2}=1$.
	Since
	$$\begin{cases}
		\sigma\tau_E\sigma^{-1}(0,b,c)=(0,c,\al^2 b) \\
		\sigma\tau_E\sigma^{-1}(a,b,c)=(a,\al^{2}c,\al^{-2}b),
	\end{cases}$$
    if $\sigma^{2}={\rm id}$, then $\sigma\tau_E\sigma^{-1}=\tau_E$. Since $\tau_E \in Z(E,\sigma)$,
    $Z(E,\sigma)=\ddPAut$.
	
	(2) By direct calculations, $(\tau_E\sigma)^2={\rm id}$, so
	$\sigma\tau_E\sigma\tau_E^{-1}=\sigma^2=\tau_E^{-1}\sigma^{-1}\tau_E\sigma$.
	By Lemma \ref{lem.ztau}, $\sigma\tau_E\sigma^{-1}\tau_E^{-1}, \tau_E^{-1}\sigma^{-1}\tau_E\sigma \in \ddPAut$
	if and only if $\sigma^6={\rm id}$. In particular, if $\sigma^6={\rm id}$, then $\sigma\tau_E\sigma^{-1},
	\sigma^{-1}\tau_E\sigma \in \ddPAut$.
	By Lemma \ref{lem.gnmn}, $M(E,\sigma)=N(E,\sigma)=\ddPAut$, hence (2) holds.
	
	(3) If $\sigma^6 \neq {\rm id}$, then $\tau_E \notin N(E,\sigma)$.
	Since $G(E) \cap N(E,\sigma)=\{ {\rm id} \}$, $N(E,\sigma)=Z(E)$, hence (3) holds.
\end{proof}
\begin{thm}\label{thm.listNEC}
	Let $A=\mathcal{A}(E,\sigma)$ be a $3$-dimensional standard AS-regular algebra except for Type EC.
	Then the following list gives $Z(E,\sigma)$ and $M(E,\sigma)$ for each type;
	\begin{center}
		{\renewcommand\arraystretch{1.1}
			{\small
				\begin{tabular}{|c|p{5.5cm}|p{5.5cm}|}
					\multicolumn{3}{c}{ \rm Table 3}
					\\[5pt]
					\hline
					{\rm Type}
					& $Z(E,\sigma)$
					& $M(E,\sigma)$
					\\ \hline\hline
					$\rm{P}$ 
					& 
					\fontsize{6pt}{10pt}\selectfont$\PGl_{3}(k)$
					& \fontsize{6pt}{10pt}\selectfont$\PGl_{3}(k)$
					\\ \hline
					$\rm{T}$ 
					& \fontsize{6pt}{10pt}\selectfont
					$\left\{
					\begin{pmatrix}
						1 & 0 & 0 \\
						0 & e & 0 \\
						g & h & 1
					\end{pmatrix} \middle| e^3=1
					\right\} \rtimes \left\langle
					\begin{pmatrix}
						0 & 1 & 0 \\
						1 & 0 & 0 \\
						0 & 0 & -1
					\end{pmatrix} \right\rangle
					$
					& \fontsize{6pt}{10pt}\selectfont
					$\left\{
					\begin{pmatrix}
						1 & 0 & 0 \\
						0 & e & 0 \\
						g & h & 1
					\end{pmatrix} \middle| e^3=1
					\right\} \rtimes \left\langle
					\begin{pmatrix}
						0 & 1 & 0 \\
						1 & 0 & 0 \\
						0 & 0 & -1
					\end{pmatrix} \right\rangle$
					\\ \hline
					$\rm{T'}$ 
					& \fontsize{6pt}{10pt}\selectfont$\left\{
					\begin{pmatrix}
						1 & 0 & 0 \\
						d & 1 & 0 \\
						d^2 & 2d & 1
					\end{pmatrix}
					\right\}$
					& \fontsize{6pt}{10pt}\selectfont
					$
					\left\{
					\begin{pmatrix}
						1 & 0 & 0 \\
						d & 1 & 0 \\
						d^2 & 2d & 1
					\end{pmatrix}
					\right\}$
					\\ \hline
					
					$\rm{CC}$ 
					& \fontsize{6pt}{10pt}\selectfont$
					\left\{
					\begin{pmatrix}
						1 & 0 & 0 \\
						0 & 1 & 0 \\
						0 & 0 & 1
					\end{pmatrix}
					\right\}
					$
					& \fontsize{6pt}{10pt}\selectfont$
					\left\{
					\begin{pmatrix}
						1 & 0 & 0 \\
						0 & 1 & 0 \\
						0 & 0 & 1
					\end{pmatrix}
					\right\}$
					\\ \hline

				\end{tabular}
				}}
			\end{center}
	\begin{center}
        {\renewcommand\arraystretch{1.1}
          {\small
                   \begin{tabular}{|c|p{5.5cm}|p{5.5cm}|}
                    \hline
                    $\rm{S}$ 
                    & \fontsize{6pt}{10pt}\selectfont$
                    \begin{cases}
	                D \rtimes \left\langle
	                \begin{pmatrix}
		            0 & 1 & 0 \\
		            0 & 0 & 1 \\
		            1 & 0 & 0
	                \end{pmatrix} \right\rangle
	                \text{ if }\,\,\,\sigma^2 \neq {\rm id}, \\
	                \ddPAut
	                \text{ if }\,\,\,\sigma^2={\rm id}.
                    \end{cases}$
                    & \fontsize{6pt}{10pt}\selectfont$
                    \begin{cases}
	                D \rtimes \left\langle
	                \begin{pmatrix}
		            0 & 1 & 0 \\
		            0 & 0 & 1 \\
		            1 & 0 & 0
	                \end{pmatrix} \right\rangle
	                \text{ if }\,\,\,\sigma^6 \neq {\rm id}, \\
	                \ddPAut
	                \text{ if }\,\,\,\sigma^6={\rm id}.
                    \end{cases}
                    $
                    \\[18pt] \hline
                    $\rm{S'}$ 
                    & \fontsize{6pt}{10pt}\selectfont$
                    \begin{cases}
	                \left\{
	                \begin{pmatrix}
		            1 & 0 & 0 \\
		            0 & e & 0 \\
		            0 & 0 & e^{-1}
	                \end{pmatrix} \middle| e \neq 0
	                \right\}
	                &\text{ if }\,\,\,\sigma^2 \neq {\rm id} \\
	                \ddPAut &\text{ if }\,\,\,\sigma^2={\rm id}.
                    \end{cases}$
                    & \fontsize{6pt}{10pt}\selectfont
                    $
                    \begin{cases}
	                \left\{
	                \begin{pmatrix}
	            	1 & 0 & 0 \\
		            0 & e & 0 \\
		            0 & 0 & e^{-1}
	                \end{pmatrix} \middle| e \neq 0
	                \right\}
	                &\text{ if }\,\,\,\sigma^6 \neq {\rm id} \\
                 	\ddPAut &\text{ if }\,\,\,\sigma^6={\rm id}.
                    \end{cases}
                    $
                    \\[18pt] \hline
					$\rm{NC}$ 
					& \fontsize{6pt}{10pt}\selectfont$
					\begin{cases}
						\left\langle
						\begin{pmatrix}
							1 & 0 & 0 \\
							0 & \ep & 0 \\
							0 & 0 & \ep^2
						\end{pmatrix}\right\rangle &\text{if}\,\,\,\sigma^2 \neq {\rm id} \\
						\ddPAut &\text{if}\,\,\,\sigma^2={\rm id}
					\end{cases}$
					& \fontsize{9pt}{13pt}\selectfont
					$
					\begin{cases}
						\left\langle
						\begin{pmatrix}
							1 & 0 & 0 \\
							0 & \ep & 0 \\
							0 & 0 & \ep^2
						\end{pmatrix}\right\rangle &\text{if}\,\,\,\sigma^6 \neq {\rm id} \\
						\ddPAut &\text{if}\,\,\,\sigma^6={\rm id}
					\end{cases}$
					\\[18pt] \hline
				\end{tabular}
		}}
	\end{center}
     	where $D:=\left\{
     \begin{pmatrix}
     	1 & 0 & 0 \\
     	0 & e & 0 \\
     	0 & 0 & i
     \end{pmatrix} \middle| ei \neq 0
     \right\}$.
\end{thm}

\begin{proof}
	By Theorem \ref{thm.ctt} and Theorem \ref{thm.62zm}, the result hold.
\end{proof}
\begin{defin}
	Let $E=\mathcal{V}(x^3+y^3+z^3) \subset \mathbb{P}^2$ so that $j(E)=0$, and define
	$$\mathcal{E}:=\{ (a,b,c) \in E \mid a^9=b^9=c^9 \} \subset E[9] \setminus E[6].$$
	In this paper, we say that a $3$-dimensional quadratic AS-regular algebra is {\it exceptional}
	if it is graded Morita equivalent to $\mathcal{A}(E,\sigma_{p})$ for some $p \in \mathcal{E}$.
\end{defin}
\begin{lem}\label{lem.MN}
	Let $A=\mathcal{A}(E,\sigma_{p})$ be a $3$-dimensional standard AS-regular algebra of Type EC
	and $\sigma_{q}\tau_{E}^{i} \in \ddPAut$ where $q \in E[3]$, $i \in \mathbb{Z}_{|\tau_E|}$.
	Then
	\begin{enumerate}
		\item $\sigma_{q}\tau_{E}^i \in Z(E,\sigma_{p})$ if and only if $p-\tau_{E}^i(p)=o$,
		\item $\sigma_{q}\tau_{E}^i \in N(E,\sigma_{p})$ if and only if $p-\tau_{E}^i(p) \in E[3]$, and
		\item $M(E,\sigma_{p})=N(E,\sigma_{p})$.
	\end{enumerate}
\end{lem}
\begin{proof}
	(1) Since $\sigma_{p}(\sigma_{q}\tau_{E}^i)\sigma_{p}^{-1}=\sigma_{q+p-\tau_{E}^{i}(p)}\tau_{E}^i$,
	$\sigma_{q}\tau_{E}^{i} \in Z(E,\sigma_{p})$ if and only if $p-\tau_{E}^{i}(p)=o$.
	
	(2) Since $\sigma_{p}(\sigma_{q}\tau_{E}^i)\sigma_{p}^{-1}=\sigma_{q+p-\tau_{E}^{i}(p)}\tau_{E}^i$,
	by \cite[Lemma 5.3]{Mo}, $\sigma_{q}\tau_{E}^{i} \in N(E,\sigma_{p})$ if and only if    
	$p-\tau_{E}^{i}(p) \in E[3]$.

	(3) In general, $M(E,\sigma_{p}) \subset N(E,\sigma_{p})$,
	so it is enough to show that $N(E,\sigma_{p}) \subset M(E,\sigma_{p})$.
	Let $\sigma_{q}\tau_E^{i} \in N(E,\sigma_{p}) \subset \ddPAut$
	where $q \in E[3]$ and $i \in \mathbb{Z}_{|\tau_E|}$.
	Since $\sigma_{p}(\sigma_{q}\tau_E^{i})\sigma_{p}^{-1}=\sigma_{q+p-\tau_E^{i}(p)}\tau^{i} \in \ddPAut$,
	$p-\tau_E^{i}(p) \in E[3]$. For any $j \ge 1$, we can write
	$$(\sigma_{q}\tau_E^{i}\sigma_{p})^{j}(\sigma_{p})^{-j}=\sigma_{r_{j}}\tau_E^{ji}$$
	where $r_{j}=\sum_{l=0}^{j-1}\tau_E^{li}(q)+\sum_{l=1}^{j}\tau_E^{li}(p-\tau_E^{(j-l)i}(p))$,
	and
	$$(\sigma_{q}\tau_E^{i}\sigma_{p})^{-j}(\sigma_{p})^{j}=\sigma_{s_{j}}\tau_E^{-ji}$$
	where $s_{j}=\sum_{l=1}^{j}(-\tau_E^{-l}(q))+\sum_{l=0}^{j-1}\tau_E^{-ji}(p-\tau_E^{(j-l)i}(p))$.
	By \cite[Lemma 4.19]{IM}, for any $j \ge 1$, $r_{j}, s_{j} \in E[3]$, so
	$$(\sigma_{q}\tau_E^{i}\sigma_{p})^{j}(\sigma_{p})^{-j}, (\sigma_{q}\tau_E^{i}\sigma_{p})^{-j}(\sigma_{p})^{j} \in \ddPAut,$$
	hence (3) holds.
\end{proof}
\begin{thm}\label{thm.listEC}
	Let $A=\mathcal{A}(E,\sigma_{p})$ be a $3$-dimensional standard AS-regular algebra of Type EC.
	Then the following table gives $Z(E,\sigma_{p})$ and $M(E,\sigma_{p})$;
	
		\begin{center}
		{\renewcommand\arraystretch{1.1}
			{\small
				\begin{tabular}{|c|c|p{5cm}|p{4cm}|}
					\multicolumn{4}{c}{\rm Table 4}
					\\[5pt]
					\hline
					{\rm Type}
					& $j(E)$
					& $Z(E,\sigma_{p})$
					& $M(E,\sigma_{p})$
					\\ \hline \hline
					& $j(E) \neq 0,12^3$
					& \fontsize{7pt}{12pt}\selectfont$
					\begin{cases}
						T[3] &\text{if}\,\,\,p \notin E[2] \\
						\ddPAut &\text{if}\,\,\,p \in E[2]
					\end{cases}$
					& \fontsize{7pt}{12pt}\selectfont$
					\begin{cases}
						T[3] &\text{if}\,\,\,p \notin E[6] \\
						\ddPAut &\text{if}\,\,\,p \in E[6] 
					\end{cases}$
					\\ \cline{2-4}
					{\rm EC} & $j(E)=0$
					& \fontsize{7pt}{12pt}\selectfont$
					\begin{cases}
						T[3] &\text{if}\,\,\,p \notin E[2] \\
						T[3] \rtimes \langle \tau_E^{3} \rangle &\text{if}\,\,\, p \in E[2]
					\end{cases}$
					& \fontsize{7pt}{12pt}\selectfont$
					\begin{cases}
						T[3] &\text{if}\,\,\,p \notin \mathcal{E} \cup E[6] \\
						T[3] \rtimes \langle \tau_E^{2} \rangle &\text{if}\,\,\,p \in \mathcal{E} \\
						T[3] \rtimes \langle \tau_E^{3} \rangle &\text{if}\,\,\,p \in E[6]
					\end{cases}$
					\\ \cline{2-4}
					& $j(E)=12^3$
					& \fontsize{7pt}{12pt}\selectfont$
					\begin{cases}
						T[3] &\text{if}\,\,\, p \notin E[2] \\
						T[3] \rtimes \langle \tau_E^2 \rangle &\text{if}\,\,\, p \in E[2] \setminus \langle (1,1,\la) \rangle \\
						\ddPAut &\text{if}\,\,\, p=(1,1,\la)
					\end{cases}$
					& \fontsize{7pt}{12pt}\selectfont
					$
					\begin{cases}
						T[3] &\text{if}\,\,\,p \notin E[6] \\
						T[3] \rtimes \langle \tau_E^2 \rangle &\text{if}\,\,\,p \in E[6] \setminus \mathcal{F} \\
						\ddPAut &\text{if}\,\,\,p \in \mathcal{F}
					\end{cases}$
					\\ \hline
		\end{tabular}
		}}
	\end{center}
where $\mathcal{F}:=\langle (1,1,\la) \rangle \oplus E[3]$.
\end{thm}

\begin{proof}
	By Lemma \ref{lem.MN} (3), it is enough to calculate $Z(E,\sigma_{p})$ and $N(E,\sigma_{p})$.
	By Lemma \ref{Ztwist} (2), $Z(E,\sigma_{p})$ was given in \cite[Proposition 4.7]{M}.
	The set of points satisfying $p-\tau_{E}^i(p) \in E[3]$ was given in \cite[Theorem 3.8]{M}.
	By Lemma \ref{lem.MN} (1) and (2), the result follows.
\end{proof}
Corollary \ref{thm.main} shows that in most cases
a twisting system can be replaced by an automorphism
to compute a twisted algebra.
\begin{cor}\label{thm.main}
	Let $A=\mathcal{A}(E,\sigma)$ be a $3$-dimensional non-exceptional standard AS-regular
	algebra. If $\sigma^6 \neq {\rm id}$ or $\sigma^2={\rm id}$, then $Z(E,\sigma)=M(E,\sigma)$, so
	${\rm Twist}_{\rm alg}(A)={\rm Twist}(A)$.
\end{cor}

\begin{proof}
	By Theorem \ref{thm.ctt} and Theorem \ref{thm.62zm}, it is enough to show the case that
	$A=\mathcal{A}(E,\sigma_{p})$ is of Type EC such that $j(E)=0$, $p \notin \mathcal{E}$, or $j(E)=12^3$.
	
	(1) $j(E)=0$, $p \notin \mathcal{E}$: Let $E=\mathcal{V}(x^3+y^3+z^3) \subset \mathbb{P}^2$.
    By Theorem \ref{thm.listEC},
    if $p \in E[2]$ or $p \notin E[6]$, then $N(E,\sigma_{p})=M(E,\sigma_{p})=Z(E,\sigma_{p})$.
    
    (2) $j(E)=12^3$: Let $E=\mathcal{V}(x^3+y^3+z^3-3\la xyz) \subset \mathbb{P}^2$ where $\la=1+\sqrt{3}$.
    By Theorem \ref{thm.listEC},
    if $p \in E[2]$ or $p \notin E[6]$, then $N(E,\sigma_{p})=M(E,\sigma_{p})=Z(E,\sigma_{p})$.
\end{proof}
Let $E \subset \mathbb{P}^2$ be a projective variety. 
For $\tau \in \Aut_{k}\,E$, we define
$$||\tau||:=\inf\{
i \in \mathbb{N}^{+} \mid \tau^i \in \uuPAut
\}.$$

\begin{cor}\label{thm.main2}
	For every $3$-dimensional non-exceptional geometric AS-regular algebra $B$,
	there exists a $3$-dimensional standard AS-regular algebra $S$ such that
	${\rm Twist}(B)={\rm Twist}_{\rm alg}(S)$.
\end{cor}

\begin{proof}
	By Lemma \ref{lem.SP}, there exists a $3$-dimensional non-exceptional standard AS-regular algebra
	$A=\mathcal{A}(E,\sigma)$ such that $\GrMod\,B \cong \GrMod\,A$. If $\sigma^6 \neq {\rm id}$, then
	${\rm Twist}(B)={\rm Twist}(A)={\rm Twist}_{\rm alg}(A)$ by Corollary \ref{thm.main},
	so we assume that $\sigma^6={\rm id}$. Set $\tau:=\sigma^2 \in \Aut_{k}\,E$ and $S:=\mathcal{A}(E,\sigma^3)$.
	Since $||\tau||=||\sigma^2||=|\sigma^6|=1$ by \cite[Theorem 3.4]{IMo},
	$\tau \in \uuPAut$. Since $\tau^{i+1}\sigma=\sigma^{2i+3}=\sigma^3\tau^i$ for every $i \in \mathbb{Z}$,
	$\GrMod\,A \cong \GrMod\,S$ by Theorem \ref{thm.Mori}.
	Since $(\sigma^{3})^2={\rm id}$, ${\rm Twist}(B)={\rm Twist}(A)={\rm Twist}(S)={\rm Twist}_{\rm alg}(S)$ by Corollary \ref{thm.main}.
\end{proof}
Example \ref{exm.EC} shows that even if $B \cong S^{\theta}$ for some $3$-dimensional quadratic Calabi-Yau
AS-regular algebra $S$, there may be no $\phi \in \GrAut_{k}\,S$ such that $B \cong S^{\phi}$.
We need to carefully choose $S$ in order that $B \cong S^{\phi}$ for some $\phi \in \GrAut_{k}\,S$.
\begin{exa}\label{exm.EC}
	Let $E \subset \mathbb{P}^2$ be an elliptic curve.
	Assume that $j(E) \neq 0,12^3$.
	We set three geometric algebras of Type EC;
	$B:=\mathcal{A}(E,\tau_{E}\sigma_{p})$, $A:=\mathcal{A}(E,\sigma_{p})$ and $S:=\mathcal{A}(E,\sigma_{3p})$
	where $p \in E[6] \setminus (E[2] \cup E[3])$.
	By \cite[Theorem 4.3]{IM2}, these algebras are $3$-dimensional quadratic AS-regular algebras.
	Moreover, $A$ and $S$ are standard.
	By \cite[Theorem 4.20]{IM}, $\GrMod\,B \cong \GrMod\,A \cong \GrMod\,S$.
	Since $|\sigma_{p}|=6$ and $|\sigma_{3p}|=2$, $M(E,\sigma_{3p})=Z(E,\sigma_{3p}) \neq Z(E,\sigma_{p})$ by Table 4, so
	${\rm Twist}(B)={\rm Twist}_{\rm alg}(S) \neq {\rm Twist}_{\rm alg}(A)$.
\end{exa}



\begin{thebibliography}{HD}
\bibitem{AS}
M.~Artin and W.~Schelter,
Graded algebras of global dimension $3$,
Adv. Math., {\bf 66} (1987), 171--216.


\bibitem{CG}
N.~Cooney and J.~E.~Grabowski,
Automorphism groupoids in noncommutative projective geometry,
arXiv:1807.06383v3.

\bibitem{F}
H.~R.~Frium,
The group law on elliptic curves on Hesse form,
Finite fields with applications to coding theory, cryptography and related areas, (Oaxaca, 2001), Springer, Berlin, (2002), 123--151.

\bibitem{G}
V.~Ginzburg,
Calabi-Yau algebras,
arXiv:0612139 (2007).

\bibitem{H}
R.~Hartshorne, {\it Algebraic Geometry},
Graduate Texts in Mathematics, No. {\bf 52}, Springer-Verlag, New York-Heidelberg, 1977.

\bibitem{HMM}
H.~Hu, M.~Matsuno and I.~Mori,
Noncommutative conics in Calabi-Yau quantum projective planes,
preprint, arXiv:2104.00221v2.

\bibitem{IM}
A.~Itaba and M.~Matsuno,
Defining relations of $3$-dimensional quadratic AS-regular algebras,
Math. J. Okayama Univ. {\bf 63} (2021), 61--86.

\bibitem{IM2}
A.~Itaba and M.~Matsuno,
AS-regularity of geometric algebras of plane cubic curves,
J. Aust. Math. Soc. {\bf 112}, no. 2 (2022), 193--217.

\bibitem{IMo}
A.~Itaba and I.~Mori,
Quantum projective planes finite over their centers,
Cand. Math. Bull., to appear.

\bibitem{M}
M.~Matsuno,
A complete classification of $3$-dimensional quadratic AS-regular algebras of Type EC,
Canad. Math. Bull. {\bf 64} (1) (2021), 123--141.

\bibitem{Mo}
I.~Mori,
Non commutative projective schemes and point schemes,
{\it Algebras, Rings and Their Representations}, World Sci., Hackensack, N.J., (2006), 215--239.

\bibitem{MU}
I.~Mori and K.~Ueyama,
Graded Morita equivalences for geometric AS-regular algebras,
Glasg. Math. J. {\bf 55} (2013), no {\bf 2}, 241--257.

\bibitem{Y}
Y.~Yamamoto,
Stabilizers of regular superpotentials in $3$-variables,
Master's thesis, Shizuoka University, (2022) (Japanese).

\bibitem{Z}
J.~J.~Zhang,
Twisted graded algebras and equivalences of graded categories,
Proc. Lond. Math. Soc. {\bf 72} (1996), 281--311.

\end{thebibliography}
\end{document}